\documentclass[13pt]{article}
\usepackage{amssymb,amsmath,amsthm}
\newtheorem{theorem}{Theorem}[section]
\newtheorem{proposition}{Proposition}[section]
\newtheorem{corollary}{Corollary}[section]
\newtheorem{lemma}{Lemma}[section]
\newtheorem{remark}{Remark}[section]

\newtheorem{conjecture}{Conjecture}[section]
\newtheorem{definition}{Definition}[section]
\begin{document}
\title{Geodesic rigidity of Levi-Civita connections admitting essential projective vector fields 
}


\author{Tianyu Ma}
\providecommand{\keywords}[1]{\textbf{\textit{Index terms---}} #1}


\maketitle

\begin{abstract}
In this paper, it is proved that a connected 3-dimensional Riemannian manifold or a closed connected semi-Riemannian manifold $M^n$ ($n>1$) admitting a projective vector field with a non-linearizable singularity is projectively flat.
\end{abstract}

\section{Introduction}
\label{sec:intro}
Let $\nabla$ be a torsion-free affine connection on a manifold $M^n$. The projective class $[\nabla]$ for $\nabla$ consists of the torsion-free affine connections on $M$ having the same unparametrized geodesics as $\nabla$. There is an equivalence of categories between projective classes and normal projective Cartan geometries on $M$ introduced by Kobayashi and Nagano in $\cite{Nagano}$, see Section $\ref{sec:prelim-cartan}$ for details. We say $[\nabla]$ is flat, if locally around any point $x\in M$, there exists a neighborhood $U_x$ of $x$ and $\nabla'\in [\nabla\vert_{U_x}]$ with $\nabla'$ having zero curvature. The projective structure $[\nabla]$ is metrizable if there is a Levi-Civita connection, induced by some metric $g$, contained in it. Two metrics on $M^n$ are projectively equivalent if their Levi-Civita connections are in the same projective class. It is well known that
\begin{align}
\label{eqn:proj-equiv}
 \overline{\nabla}\in [\nabla]\quad \Longleftrightarrow \quad \overline{\nabla}=\nabla+\eta\otimes Id+ Id \otimes \eta,\ \eta\in \Gamma(T^*M).
 \end{align}
 Affine connections are not torsion-free in general, but for any affine connection $\nabla$, we can always find a torsion-free connection with the same unparametrized geodesics. The torsion $\mathrm{Tor}$ of $\nabla$ is a section of $\wedge^2T^*M\otimes TM$. Then $(\nabla-\mathrm{Tor})$ is a torsion-free connection defining the same unparametrized geodesics as $\nabla$. Hence, there is no loss in generality to work only with torsion-free affine connections, whose projective equivalence is characterized by Equation $\eqref{eqn:proj-equiv}$.
\\
\\
 Let $X$ be a vector field on $M$. Let $\phi^t$ be the flow generated by $X$. Then $X$ is a projective vector field for $\nabla$ if $\phi^t$ preserves the unparametrized geodesics defined by $\nabla$. Denote $\mathcal{L}_X\nabla$ the Lie derivative of $\nabla$ with respect to $X$. Then this is equivalent to:
 $$\textbf{Trace free part}(\mathcal{L}_X\nabla)=0.$$
 The projective vector field $X$ is $affine$ for $\nabla$ if $\mathcal{L}_X\nabla=0$. It is $essential$ if it is not affine for any connection in $[\nabla]$.\\
 \\
It is a classical topic to study projective structures induced by Levi-Civita connections. Some classical results have been obtained by mathematicians like Dini, Levi-Civita, Weyl, and Solodovnikov. One can refer to Theorems 7-10 from $\cite{matrem}$ for their results. The local description of projectively equivalent metrics is well understood by Bolsinov and Matveev in $\cite{local}$, and $\cite{splitting}$ in terms of BM structures.\\
\\
Given a projective structure $[\nabla]$ on some manifold $M^n$, how its projective transformation group or Lie algebra determines the projective structure $[\nabla]$ has been an interesting topic. For example, we may ask what additional assumption on the projective transformation group or algebra is necessary to deduce that the projective structure is flat on the manifold or some special subsets. Sometimes it turns out $[\nabla]$ is determined by assumptions less than expected. Concerning the global theory of projective structures, we have the following projective Lichnerowicz-Obata conjecture:
\begin{conjecture}
Let $G$ be a connected Lie group acting on a complete connected or closed connected semi-Riemannian manifold $(M^n,g)$ by projective transformations. Then either $G$ acts on $M$ by affine transformations, or $(M^n,g)$ is Riemannian with positive constant sectional curvature.
\end{conjecture}
\noindent
The conjecture above implies non-affine complete projective vector fields cannot exist for non-flat projective structures induced by closed or complete connected $(M^n,g)$. The open cases for this conjecture are when $g$ is an indefinite metric, and $D(M^n,g)$ is precisely two, where $D(M^n,g)$ is the degree of mobility of $g$ on $M$ defined in Definition \ref{def:BM}. (In addition to the Riemannian case, this conjecture has also been proved for the case $(M,g)$ being a closed connected Lorentzian manifold, see $\cite{matcproj}$.) One may refer to the main theorems in $\cite{matpse},\cite{matrem}$ and $\cite{matcproj}$ for details. In the local theory of projective structures, whether there is a result analogous to the conjecture above for locally defined metrizable projective structures is still open in general.
\\
\\
Let $[\nabla]$ be a metrizable projective structure admitting a projective vector field $X$ with a non-linearizable singularity $x$. This means $X$ is not linear in any coordinate system around $x$. In this paper, we give proofs to Theorem $\ref{thm:rem}$ and $\ref{thm:closed}$, which concern the rigidity of such projective structures. The assumptions in these theorems arise from the generalization of results obtained for projective geometries in \cite{essential} by Nagano and Ochiai, and analogous results for conformal geometries by Frances and Melnick in $\cite{Frances}$ and $\cite{mel}$. Let $X$ be a vector field on $M$. We say $X$ vanishes exactly at order 2 at $x\in M$, denoted by $O(X,x)=2$, if $X$ has a zero 1-jet and a non-zero 2-jet at $x$. In $\cite{essential}$, the following theorem is proved.\\
\begin{theorem}[Nagano, Ochiai \cite{essential}]
\label{thm:ochiai}
Let $X$ be a projective vector field for a closed connected Riemannian manifold $(M^n,g)$ with $n\geq 3$. Suppose there exists $x\in M$ such that $O(X,x)=2$. It follows that $(M^n,g)$ is isometric to either $\mathbb{S}^n$ or $\mathbb{RP}^n$ with their respective standard metrics.
\end{theorem}
\noindent
One may ask whether a generalization of this theorem will hold for semi-Riemannian closed connected manifolds, with the weaker assumption that $X$ is non-linearizable at $x$. Obviously, this generalization of Theorem $\ref{thm:ochiai}$ follows from the projective Lichnerowicz-Obata conjecture.\\
\\
The dynamics of a projective vector field near its singularity can lead to theorems on the rigidity of projective structures. For example, since the projective Weyl curvature is invariant under projective maps, Nagano and Ochiai obtained the following result (See Lemma 5.6 of $\cite{essential}$ for details), which is the main ingredient to prove Theorem $\ref{thm:ochiai}$: If an affine connection $\nabla$ defined on $M^n$ with $n\geq 3$ admits a projective vector field $X$ such that $O(X,x)=2$ at some point $x$, then $[\nabla]$ is projectively flat near $x$. \\
\\
  Suppose that $x$ is a non-linearizable singularity of a projective vector field $X$. Then on some special subsets containing $x$, the flow $\phi^t$ generated by $X$ admits dynamics similar to the case $O(X,x)=2$. This may imply $X$ admits a non-linearizable singularity at $x$ is a good substitution for the assumption $O(X,x)=2$.\\
\\
Projective and conformal structures are both $\vert 1\vert$-graded parabolic geometries in terms of Cartan geometries. In conformal geometries we have the following result from $\cite{mel}$.
\begin{theorem}[C. Frances, K. Melnick $\cite{mel}$]
\label{thm:melnick}
Let $X$ be a conformal vector field for a semi-Riemannian manifold $(M^n,g)$ with $n\geq 3$ with a singularity $x$. If the 1-parameter group $\lbrace(D\phi^t_X)_x:t\in \mathbb{R}\rbrace$ is bounded, one of the following is true:
\begin{itemize}
\item
There exists a neighbourhood $V$ of $x$ on which $X$ is complete and generates a bounded flow. In this case, it is linearizable.
\item
There is an open set $U_0\subset M$, with $x\in \overline{U_0}$ such that $g$ is conformally flat on $U_0$.
\end{itemize}
\end{theorem}
\noindent
In terms of the local theory of projective structures, one can expect a statement analogous to Theorem $\ref{thm:melnick}$ to hold for projective geometries. Let $X$ be a projective vector field for $[\nabla]$ vanishing at $x$. The minimal conditions for the projective class $[\nabla]$ being flat near $x$ are still open. 
\\
\\
In this paper, the following theorem on projective geometries induced by Riemannian metrics is proved.
\begin{theorem}
\label{thm:rem}
Let $(M^n,g)$ with $n\geq 3$ be a connected Riemannian manifold admitting a projective vector field $X$. Suppose $X$ vanishes at $o\in M$, and $X$ is not linearizable at $o$. We have $D(M^n,g)$ is at least 3. When $n=3$, this implies $g$ has constant sectional curvature.
\end{theorem}
\noindent
For closed and connected manifolds, the following generalization of Theorem $\ref{thm:ochiai}$ by Nagano and Ochiai is proved.
\begin{theorem}
\label{thm:closed}
Let $(M^n,g)$ with $n>1$ be a closed connected semi-Riemannian manifold. Suppose $X$ is a projective vector field for $(M,g)$ vanishing at $o\in M$. If $X$ is not linearizable at $o$, then $g$ is Riemannian with constant positive sectional curvature.
\end{theorem}
\noindent
After deriving the proof, we discovered that the part of the proof of Theorem \ref{thm:closed} in Section \ref{sec:closed} is analogous to Section 9.2 of $\cite{zig}$.\\
\\
Note that Theorem $\ref{thm:closed}$ is just a sub-case of the projective Lichnerowicz Conjecture. We will show in Proposition $\ref{essential}$, if $X$ has a non-linearizable singularity, then the flow $\phi^t$ generated by $X$ acts on $(M^n,g)$ by non-affine transformations. Thus if the projective Lichnerowicz Conjecture is true, Theorem $\ref{thm:closed}$ follows trivially from it.
\section{Preliminaries and Backgrounds}
\label{sec:prelim}
\subsection{General theory for projective structures in the view of Cartan geometries}
\label{sec:prelim-cartan}
The definition of a Cartan geometry is given below, since it is important for this paper. Let $\hat{G}$ be a Lie group, and $G'$ is a closed subgroup of $\hat{G}$. Denote $\hat{\mathfrak{g}},\mathfrak{g}'$ their Lie algebras, respectively. The definition of a Cartan geometry modelled on $(\hat{\mathfrak{g}},\mathfrak{g}')$ with the structure group $G'$ is as follows.
 \begin{definition}
 \label{def:cartan}
A Cartan geometry modelled on $(\hat{\mathfrak{g}},\mathfrak{g}')$ with the structure group $G'$ is a triple $(M,B,\omega)$. Here $B$ is a $G'$ principal bundle over $M$, and $\omega$ a Cartan connection, that is a $\hat{\mathfrak{g}}$-valued 1-form satisfying the following conditions:
 \begin{itemize}
 \item $\forall b\in B$, the map $\omega_b:T_bB\rightarrow \hat{\mathfrak{g}}$ is an isomorphism.
 \item $\forall g\in G'$, $R_g^*\omega=Ad(g^{-1})\omega$. Here $R_g$ is the right translation of the principal $G'$-bundle $B$ by $g$.
 \item $\forall b\in B,\ \forall \tilde{g}\in\mathfrak{g}'$,\ we have $\omega\left(\dfrac{d}{dt}\vert_{t=0}b\exp(t\tilde{g})\right)=\tilde{g}.$
 \end{itemize}
 \end{definition}
 \noindent
We have $\kappa =d\omega+\dfrac{1}{2}[\omega,\omega]$ is the curvature of this Cartan geometry. The Cartan geometry is flat if $\kappa$ vanishes. Let $\omega_{\hat{G}}$ be the Maurer-Cartan form on $\hat{G}$ (Refer to Page 98 of $\cite{Sharpe}$ for the definition). The triple $(\hat{G}/G',\hat{G},\omega_{\hat{G}})$ defines a flat Cartan geometry. This is a flat model for $(\hat{\mathfrak{g}},\mathfrak{g}')$ with the structure group $G'$. If $(M,B,\omega)$ is flat, then it is locally isomorphic to $(\hat{G}/G',\hat{G},\omega_{\hat{G}})$, see Theorem 6.1 of Chapter 3 of $\cite{Sharpe}$.\\
 \\
 We give the definition of exponential maps in Cartan geometries.
 \begin{definition}
 \label{def:exp}
 Suppose $(M,B,\omega)$ is a Cartan geometry modelled on $(\hat{\mathfrak{g}},\mathfrak{g}')$. Given any $v\in\hat{\mathfrak{g}}$, we have $\omega^{-1}(v)$ is a vector field on $B$. Denote $\Phi_v$ the flow generated by $\omega^{-1}(v)$. The exponential map of $\omega$ at $b\in B$ is defined as $\exp_b(v)=\Phi_v(1,b)$, wherever it is well defined. Thus $\exp_b$ gives a local diffeomorphism between a neighbourhood of $0$ of $\hat{\mathfrak{g}}$ and a neighbourhood of $b$.
 \end{definition}
 \noindent
 The projective classes on $M$ can be described in terms of Cartan geometries by the following. The group $G=PGL(n+1,\mathbb{R})$ acts on $\mathbb{RP}^n$ transitively. Choose $e_0=[1,0,\cdots,0]\in \mathbb{RP}^n$, and let $H$ be its stabilizer. Denote $\mathfrak{g},\mathfrak{h}$ the Lie algebras of $G$ and $H$, respectively. Then we have the following identification (see Page 234 of $\cite{essential}$):
\begin{eqnarray}
\label{eqn1}
\mathfrak{sl}(n+1,\mathbb{R})=\mathfrak{g}=\mathfrak{g}_{-1}\oplus \mathfrak{g}_0 \oplus \mathfrak{g}_1 \simeq \mathbb{R}^n \oplus GL(n,\mathbb{R}) \oplus (\mathbb{R}^n)^* ,\quad \mathfrak{h}=\mathfrak{g}_0 \oplus \mathfrak{g}_1.
\end{eqnarray}
Note that the standard Euclidean metric gives an identification $\mathbb{R}^n\simeq (\mathbb{R}^n)^*$. The identification is given by
\begin{eqnarray}
\label{eqn2}
u \oplus A \oplus v^* \mapsto \begin{bmatrix}
-\dfrac{1}{n+1}Tr(A) && v^T\\
u && A-\dfrac{1}{n+1}Tr(A)\cdot Id
\end{bmatrix}\in \mathfrak{sl}(n+1,\mathbb{R}).
\end{eqnarray}
The following is the standard chart of $\mathbb{RP}^n$ near $e_0$:
$$\label{i0}i_0:\ [x_0,\cdots,x_n]\mapsto (\dfrac{x_1}{x_0},\cdots,\dfrac{x_n}{x_0})$$
In this chart $i_0$, any $h\in H$ is a local diffeomorphism at $0\in \mathbb{R}^n$ with $h(0)=0$ . If $f$ is a local diffeomorphism at $0\in \mathbb{R}^n$ with $f(0)=0$, let $J^k(f)(0)$ be its $k$-jet at the origin. Define $G^k(n)$ to be the $k$-jet at $0$ of all such functions. Clearly elements in $G^k(n)$ form a group. Since every $h\in H$ is such a diffeomorphism in the standard chart $i_0$, define the subgroup $H^2(n)$ of $G^2(n)$:
$$H^2(n)=\lbrace J^2(h)(0):h\in H\rbrace.$$
\\
	The above in fact gives an identification $H\cong H^2(n)\cong GL(n,\mathbb{R})\ltimes \mathbb{R}^n$. Since $G^1(n)\cong GL(n,\mathbb{R})$ is induced by invertible linear maps, we can identify $G^1(n)$ with the subgroup $GL(n,\mathbb{R})$ of $H^2(n)$. Let $F^k(M)$ be the $k$th order frame bundle of $M$, which is a $G^k(n)$ principal bundle. We have $F^2(M)$ is a $G^2(n)$ principal bundle. We can take $F^2(M)$ as a sub-bundle of $F^1(F^1(M))$. Denote $\theta$ the canonical form on $F^1(F^1(M))$, which is a $\mathfrak{gl}_n(\mathbb{R})\bigoplus \mathbb{R}^n$ valued 1-form. It follows that $\theta\vert_{F^2(M)}$ has the following decomposition:
  $$\theta=\theta_i+\theta^i_j,\ \theta_i\in \Gamma(\textrm{Hom}(T(F^2M),\mathbb{R}^n)),\ \theta^i_j\in\Gamma(\textrm{Hom}(T(F^2M),\mathfrak{gl}_n(\mathbb{R}))).$$
Here $\theta=\theta_i+\theta^i_j$ is the canonical form on $F^2(M)$. One can refer to Page 224 of $\cite{Nagano}$ for a more precise definition.\\
\\
 A projective Cartan geometry on $M$ is a Cartan geometry $(M,B,\omega)$ modelled on the pair $(\mathfrak{g},\mathfrak{h})$. It is normal if the components of its curvature $\kappa$ satisfy Equation (2) and (3) of $\cite{Nagano}$. Under the identification given by Equations \eqref{eqn1} and \eqref{eqn2}, we have by Proposition 3 of $\cite{Nagano}$, on any $H^2(n)$ sub-bundle $P$ of $F^2(M)$, there is a unique normal projective Cartan connection $\omega=\omega_i+\omega^i_j+\omega^i$  with $\omega_i=\theta_i$, and $\omega^i_j=\theta^i_j$. We call this connection the normal projective Cartan connection associated to $P$. 
 \\
 \\
 Now we show how the $H^2(n)$ and $GL_n$ reductions of $F^2(M)$ correspond exactly to projective structures and torsion-free affine connections on $M$, respectively.\\
\\
  First we give the following way of identifying torsion-free affine connections on $M^n$ with $GL_n$ sub-bundles of $F^2(M)$. \\
  \\
   Given a torsion-free affine connection $\nabla$, $\forall x\in M$, the exponential map of $\nabla$ at $x$, denoted as $\exp^{\nabla}_x$, is a map:
  $$\exp^{\nabla}_x:U\subset T_xM \rightarrow M,\ 0\mapsto x$$
    Here $U$ is an open set of $T_xM$ containing the origin.\\
    \\
     We define a bundle inclusion $i_{\nabla}: F^1(M)\rightarrow F^2(M)$ as follows. Any $p\in F^1(M)$ in the fibre of $x$ can be uniquely identified with a linear map $\tilde{p}:\mathbb{R}^n\rightarrow T_xM$. Then we define
    $$i_{\nabla}(p)=J^2(\exp^{\nabla}_x\circ\tilde{p})(0),\ \forall p\in F^1(M).$$
    Let $F^2_1(M)=F^2(M)/GL_n(\mathbb{R})$, and $\pi^2_1:F^2(M)\rightarrow F^1(M)$ be the canonical projection. Then the $G^1(n)$ reductions of $F^2(M)$ correspond exactly to sections of $F^2_1(M)$. Notice that every section $\Gamma$ of $F^2_1(M)$ induces a unique natural bundle inclusion:
  $$\gamma_{\Gamma}:F^1(M)\rightarrow F^2(M),\quad \pi^2_1\circ \gamma_{\Gamma}=id.$$
 The identification $\nabla \mapsto i_{\nabla}$ in fact gives a 1-1 correspondence between torsion-free affine connections on $M$ and $GL_n$ reductions of $F^2(M)$ by the following summary of Proposition 10 and 11 of $\cite{Nagano}$.
\begin{theorem}[Nagano, Kobayashi \cite{Nagano}]
\label{thm:kobayashi}
 There is a 1-1 correspondence between torsion-free affine connections on $M$ and reductions from $F^2(M)$ to $F^1(M)$ given by the mapping $\nabla \mapsto i_{\nabla}$. Let $\theta=\theta_i+\theta^i_j$ be the canonical form on $F^2(M)$ as usual. For a torsion-free connection $\nabla$, denote $\Gamma$ the corresponding section of $F^2_1(M)$, then the following holds:
\begin{itemize}
\item The natural bundle inclusion $\gamma_{\Gamma}$ is exactly $i_{\nabla}$.
\item $(i_{\nabla})^*\theta^i$ is the canonical form on $F^1(M)$.
\item $(i_{\nabla})^*\theta^i_j$ is the connection form for $\nabla$.
\end{itemize}
\end{theorem}
The 1-1 correspondence between the projective structures on $M$ and $H^2(n)$ reductions of $F^2(M)$ is given by the following: For every torsion-free connection $\nabla$ on $M$, the map $i_{\nabla}$ gives a $GL_n$ reduction of the $G^2(n)$-principal bundle $F^2(M)$. Since $GL_n(\mathbb{R})\ltimes\mathbb{R}^n\cong H^2(n)\leq G^2(n)$, it induces a $H^2(n)$ sub-bundle $P(\nabla)$ of $F^2(M)$. From Proposition 12 of $\cite{Nagano}$, we have $P(\nabla)=P(\overline{\nabla})$ if and only if $\nabla$ and $\overline{\nabla}$ are projectively equivalent. Here $P(\nabla)$, along with its associated normal projective Cartan connection, is called the projective Cartan geometry associated to $[\nabla]$.
\subsection{Dynamics of projective vector fields near singularities}
\label{sec:prelim-dynamics}
 Every projective vector field $X$ on $M$ for $\nabla$ can be uniquely lifted to a vector field $\tilde{X}$ on $P=P(\nabla)$ such that $\mathcal{L}_{\tilde{X}}\omega=0$.
 \begin{definition}
 \label{def:infiauto}
 Let $(M,B,\omega)$ be a Cartan bundle. If $\tilde{X}\in \chi(B)$ is the lift of some vector field $X$ on $M$ such that $\mathcal{L}_{\tilde{X}}\omega=0$, then $\tilde{X}$ is called an infinitesimal automorphism of the Cartan bundle.
 \end{definition}
 \noindent
For the flat model $(\mathbb{RP}^n,G,\omega_G)$, the infinitesimal automorphisms are just right invariant vector fields on $G$.\\
\\
Given any torsion-free connection $\nabla$ on $M^n$, set $P=P(\nabla)$, and let $\omega$ be the normal projective Cartan connection associated to $P$. Denote $\pi:P\rightarrow M$ the standard projection. If a projective vector field $X$ vanishes at $o\in M$, we have $\forall p\in \pi^{-1}(o)$,  $\omega(\tilde{X})(p)\in \mathfrak{h}$. We can prove the following local result:
\begin{proposition}
\label{essential}
Let $X$ be a projective vector field for $(M,\nabla)$. Assume $X_o=0$ for some $o\in M$. Then the following are equivalent:
\begin{itemize}
\item $X$ is linearizable at o.
\item There exist a neighbourhood $U$ of $o$ and a torsion-free affine connection $\nabla'\in[\nabla\vert_U]$ such that $X$ is an affine vector field for $\nabla'$.
\end{itemize}
\end{proposition}
\noindent
To prove the proposition above, we need the following. Denote $\omega$ the normal projective Cartan connection associated to $P=P(\nabla)$ as before. Fix any $p$ in the fibre of $o$, and let $\exp_p$ be the exponential map of $\omega$ at $p$. Then there is a small neighbourhood $U$ of $0\in \mathfrak{g}_{-1}\simeq\mathbb{R}^n$ such that $\sigma_p=\pi\circ \exp_p:U\rightarrow M$ gives a local coordinate of $M$ at $o$. We call such coordinates the projective normal coordinates of $P(\nabla)$ at $o$ with respect to $p$. The local section $\exp_p(U)$ gives a $GL_n$ sub-bundle of $P$ over $\sigma_p(U)$. Then it induces an affine connection $\nabla_U\in[\nabla\vert_U]$ near $o$. By Theorem $\ref{thm:kobayashi}$, $\sigma_p$ is also a normal coordinate for the affine connection $\nabla_U\in[\nabla\vert_U]$ at $o$.
\begin{lemma}
\label{lemma:normalform}
Suppose $X$ is a projective vector field for $\nabla$ such that $X_o=0$. Let $P=P(\nabla)$, and define $\omega$ on $P$ as before. Choose any $p\in \pi^{-1}(o)$, then in the projective normal coordinate $\sigma_p$ of $P$ with respect to $p$, regardless of the choice of the normal projective Cartan connection $\omega$ induced by the projective structure $[\nabla]$, the form of $\phi^t$ in the local coordinate $\sigma_p$ is uniquely determined by the value of $\omega(\tilde{X})(p)$ in the following sense:\\
\\
For any torsion-free affine connection $\widehat{\nabla}$ admitting a projective vector field $Y$ vanishing at $\hat{o}$, denote $\widehat{\omega}$ the associated normal projective Cartan connection of $P(\widehat{\nabla})$. If $\exists \hat{p}\in \pi^{-1}(\hat{o})$ such that $\omega(\tilde{X})(p)=\widehat{\omega}(\tilde{Y})(\hat{p})\in\mathfrak{h}$, then the flow $\phi^t_Y$ in the coordinate $\sigma_{\hat{p}}$ has the same form as $\phi^t$ in $\sigma_p$.
\end{lemma}
\begin{proof}
Let $\tilde{X}$ be the lift of $X$ to $P$ such that $\mathcal{L}_{\tilde{X}}\omega=0$. Because $X_o=0$, we have $\omega(\tilde{X})(p)=v_h\in \mathfrak{h}$. Define the following identification along fibres over $o$:
$$\Delta:H\rightarrow pH,\quad h\mapsto ph.$$
It follows that $\Delta^*\omega\vert_{\pi^{-1}(o)}$ is the Maurer–Cartan form $\omega_H$ on $H$, by Definition \ref{def:cartan}. Let $X_h$ be a right-invariant vector field on $G$ with $\omega_G(X_h)(1)=v_h$. Note that $\omega_G(X_h)\vert_H\in \mathfrak{h}$, and $\mathcal{L}_{X_h}\omega_G=0$. It follows that $\Delta_*(X_h)=\tilde{X}\vert_{\pi^{-1}(o)}$.\\
\\
 Denote $\Phi$ the flow generated by $\tilde{X}$ on $P$, so $\Phi$ projects to a flow $\phi^t$ on $M$ fixing $o$. We have $\Phi(t,p)=ph(t)$, where the function $h(t)=\exp(tv_h)$. It is evident the function $h(t)$ depends only on $v_h$. Fix any $t_0\in \mathbb{R}$ and $v\in \mathfrak{g}_{-1}=\mathbb{R}^n$, and define the curve $l(s)=\exp_p(sv)$. Note that $\pi\circ l(s)$ is a geodesic of $[\nabla]$. Because $\mathcal{L}_{\tilde{X}}\omega=0$, we have the following: $$l_{t_0}(s):=\Phi(t_0,l(s))=\exp_{ph(t_0)}(sv).$$
 We also obtain
 $$\phi^{t_0}\circ \pi\circ l=\pi\circ l_{t_0}=\pi\circ R_{h(t_0)^{-1}}\circ l_{t_0}.$$
By the axioms of the Cartan connections, we get $$R_{h(t_0)^{-1}}\circ l_{t_0}(s)=\exp_p(s(Ad(h(t_0)(v)))).$$
Define $v'=Ad(h(t_0)(v))$, then $v'$ is totally determined by value of $v$ and $h(t_0)$. We define the curve $$f(s):=R_{h(t_0)^{-1}}\circ l_{t_0}.$$  Because $\pi \circ l(s)$ is geodesic of $[\nabla]$, one has $\pi\circ f(s)$ is also a geodesic of $[\nabla]$. Denote $v'_{-1}$ the $\mathfrak{g}_{-1}$ component of $v'$. One has $\pi \circ f(s)$ and $\pi\circ \exp_p(sv'_{-1})$ are geodesics for $[\nabla]$ with the same initial condition. It follows that  on a small interval $I$ containing $0$, $f(s):I\rightarrow P$ can be written in the following form:
$$f(s)=\exp_p(r(s)v'_{-1})g(s),\quad r(s):I\rightarrow \mathbb{R},\ g(s):I\rightarrow H.$$
$$r(0)=0,\ g(0)=1.$$
Differentiating the equation, we obtain
$$v'=\omega(\dfrac{df}{ds})=Ad(g(s)^{-1})(r'(s)v'_{-1})+\omega_H(g'(s)).$$
\\
Given a pair of functions $\lbrace r(s),g(s)\rbrace$, whether it is a solution to this equation depends only on $v'$, independent of the connection $\omega$. On the other hand, the definition of the exponential map implies that the solution $\lbrace r(s),g(s)\rbrace$ satisfying the condition $g(0)=1$ and $r(0)=0$ is unique. Note that $v'$ and $v'_{-1}$ only depend on $v$ and $h(t_0)$. It follows from the uniqueness that $\lbrace r(s),g(s)\rbrace$ depends only on $v$ and $h(t_0)$. In particular, the functions $r(t)$ and $v'\in \mathbb{R}^n$ depend only on the parameters $v,v_h,t_0$, regardless of the connection $\omega$. Given any two projective connections $\omega$ and $\omega'$ on the $H^2(n)$ bundle $P$, as long as the parameters $v,v_h,t_0$ are the same, we get the same the function $r(t)$ and $v'\in \mathbb{R}^n$. It follows that $\phi^{t_0}$ in the projective normal coordinates of $P$ with respect to $p$ depends only on $h(t_0)$. Let $Y$ be a projective vector field for $\widehat{\nabla}$ vanishing at $\hat{o}$ as in Lemma $\ref{lemma:normalform}$. Hence, there exists $\hat{p}\in \pi^{-1}(\hat{o})$ such that $\omega(\tilde{X})(p)=\widehat{\omega}(\tilde{Y})(\hat{p})$ implies $\phi^t_Y$ in $\sigma_{\hat{p}}$ has the same form as $\phi^t$ in $\sigma_p$. This completes the proof.
\end{proof}
\noindent
Suppose $X$ is a projective vector field for $(M,\nabla)$ vanishing at $o$, and fix any $p\in\pi^{-1}(o)$. As before, choose some right invariant vector field $\tilde{Y}$ on $G$ such that $\omega_G(\tilde{Y})(1)=\omega(\tilde{X})(p)\in \mathfrak{h}$, and let $Y$ be the projection of $\tilde{Y}$ on $\mathbb{RP}^n$. Then $X$ in the projective normal coordinates of $P$ with respect to $p$ has the same form of $Y$ in the projective normal coordinates of the flat model with respect to $1\in G$. Note the algebra of the projective vector fields has maximal dimension on the flat model. Thus by computations on the flat model, we obtain all possible forms of projective vector fields with a singularity at $o$ in the projective normal coordinates of $P$ with respect to $p$.
\begin{lemma}
\label{lemma:linear}
Let $X$ be a projective vector field for $(M,\nabla)$ with $X_o=0$. For any $p\in \pi^{-1}(o)$, $X$ has the following form in the projective normal coordinates of $P(\nabla)$ with respect to $p$. 
$$X_x=Ax+\langle w,x\rangle x,\quad A\in M_n(\mathbb{R}),\ w\in \mathbb{R}^n.$$
In addition, $X$ is linearizable if and only if $w\in \mathrm{Im}A^T$.
\end{lemma}
\begin{proof}
Let $X$ be a projective vector field for $(M,\nabla)$ such that $X_o=0$, and choose any $p\in \pi^{-1}(o)$. First we show $X$ has the form: $X_x=Ax+\langle w,x\rangle x$ in the projective normal coordinates of $P(\nabla)$ with respect to $p$. By Lemma \ref{lemma:normalform} and the argument in the previous paragraph, we only need to show for the flat bundle $P=(\mathbb{RP}^n,G,\omega_G)$, $X$ is in this form in the projective normal coordinates with respect to $p=1\in G$. In this case, the exponential map $\exp_p$ gives the canonical coordinate $i_0^{-1}$ of $\mathbb{RP}^n$ at $e_0$ defined on Page $\pageref{i0}$. The projective vector fields fixing $o=e_0\in \mathbb{RP}^n$ are induced by linear vector fields in $\mathbb{R}^{n+1}$ fixing the line $e_0$. Projecting these vector fields to $\mathbb{RP}^n$, we get $X$ has the form $X_x=Ax+\langle w,x \rangle x$ in the projective normal coordinates with respect to $p$.\\
\\
 Next we show $X$ in this form is linearizable if and only if $w\in \mathrm{Im}A^T$. If $w\notin \mathrm{Im}A^T$, we have $w=w_k+w'$ with $w_k\neq 0$, where $w_k\in \mathrm{Ker}A$ and $w'\in \mathrm{Im}A^T$. Denote $\phi^t$ the flow generated by $X$ as usual. In the projective normal coordinates of $P(\nabla)$ with respect to $p$, for some small interval $I$ containing $0$, we have $$\phi^t(sw_k)=\dfrac{s}{1+tas} w_k,\quad s\in I,\ a\neq 0.$$
Note that $D\phi^t(o)(w_k)=w_k\neq 0$. Without loss of generality, we can assume $a>0$. For $s>0$, we have $\dfrac{s}{1+tas}\rightarrow 0$ as $t\rightarrow +\infty$. Then $X$ is not linearizable by Lemma 4.6 of $\cite{Frances}$. Conversely, if $w\in \mathrm{Im}A^T$, the calculation in Remark $\ref{rmk1}$ below shows that one can find $p'\in\pi^{-1}(o)$ such that $X_x=(A_{p'})x$ in the projective normal coordinates with respect to $p'$. Hence it is linearizable.
\end{proof}
\begin{remark}
\label{rmk1}
To simply the calculations later, suppose $X$ vanishes at $o$. Note that for any $A\in M_n(\mathbb{R})$, we can write $\mathbb{R}^n=Im(A^T)\bigoplus \mathrm{Ker}A$. Then for any $p\in \pi^{-1}(o)$, this decomposition of $\mathbb{R}^n$ gives
$$S_p=\omega(\tilde{X})(p)=\begin{bmatrix}
-b && w_i^T A+w_k^T\\
0 && B
\end{bmatrix}\in \mathfrak{sl}_{n+1}(\mathbb{R}).$$
$$\ A=B+b\cdot Id,\  w_k\in \mathrm{Ker}A.$$
Define $C=\begin{bmatrix}
1 && -w_i^T\\
0 && Id
\end{bmatrix}$, we have $CS_pC^{-1}=\begin{bmatrix}
-b && w_k^T\\
0 && B
\end{bmatrix}$. In other words, given any local coordinate $\tilde{\sigma}:U\subset \mathbb{R}^n\rightarrow M$, with $\tilde{\sigma}(0)=o$, we can choose some $\tilde{p}\in\pi^{-1}(o)$ such that the projective normal coordinate $\sigma_{\tilde{p}}$ with respect to $\tilde{p}$ of $P$ satisfies:
$$J^1(\tilde{\sigma})(0)=J^1(\sigma_{\tilde{p}})(0),\quad ((\sigma_{\tilde{p}}^{-1})_*X)_x=Ax+\langle w,x \rangle x,\ w\in \mathrm{Ker}A.$$
\end{remark}
\noindent
With the results above, we can prove Proposition \ref{essential}.
\begin{proof}[Proof of Proposition $\ref{essential}$]
By Remark \ref{rmk1}, we can always choose some $p\in \pi^{-1}(o)$ such that in the projective normal coordinate $\sigma_p$ of $P(\nabla)$ with respect to $p$, $X$ has the following form:$$X_x=Ax+\langle w,x\rangle x,\ w\in \mathrm{Ker}A.$$ If $X$ is linearizable at $o$, we have $w\in \mathrm{Im}A^T$ by Lemma \ref{lemma:linear}. It follows that $w=0$, then $X$ is linear in $\sigma_p$ and $\omega(\tilde{X})(p)\in\mathfrak{g}_0$. According to Theorem $\ref{thm:kobayashi}$ by Nagano and Kobayashi, the $GL_n$ sub-bundle $P_1$ of $F^2(M)$ induced by the local section $\exp_p(\mathfrak{g}_{-1})$ corresponds to a connection $\nabla'$ projectively equivalent to $\nabla$ locally defined near $o$. Then $P_1\subset P(\nabla)$ is invariant by the flow of $\tilde{X}$ because $\omega(p)(\tilde{X})\in \mathfrak{g}_0$. We have
 $$\tilde{X}\vert_{P_1}\subset TP_1,\quad \mathcal{L}_{\tilde{X}}\theta^i_j\vert_{P_1}=\mathcal{L}_{\tilde{X}}\omega^i_j\vert_{P_1}=0.$$
  Hence, $X$ is affine for $\nabla'$ by the last statement of Theorem $\ref{thm:kobayashi}$.
The converse is trivial as affine vector fields of $\nabla'$ vanishing at $o$ are clearly linear in the normal coordinates of $\nabla'$ at $o$.
\end{proof}
\noindent
Suppose that $X$ is a non-linearizable projective vector field for $(M,\nabla)$ vanishing at $o\in M$. For each $a>0$, we can choose a neighbourhood $U_a$ of $o$ such that $\phi^t$ is well defined on $U_a$ for $t\in I=[-a,a]$. One has on $U_a$, $\nabla_t=\phi^t_*\nabla$ is projectively equivalent to $\nabla$ for $t\in I$. If $\gamma(s)$ is a geodesic segment for $\nabla$ contained in $\phi^{t_0}(U_a)$ with $t_0\in I$, then $\phi^{-t_0}\circ \gamma(s)$ is a geodesic segment on $U_a$ for $\nabla_{t_0}$. This leads to the following:
\begin{corollary}
\label{cor:2.1.1}
Let $X$ be a projective vector field for $(M,\nabla)$ admitting a non-linearizable vanishing point $o\in M$. Then for each $t\neq0$, we have $$\nabla_t=\nabla+\eta_t\otimes Id +Id\otimes \eta_t,\quad (\eta_t)_o\neq 0.$$
\end{corollary}
\begin{proof}
Suppose that $\eta_{t_0}(o)=0$ for some $t_0\neq 0$. The connection $\nabla$ induces a $GL_n$ sub-bundle $P_1$ of $P(\nabla)$. Choose $p\in \pi^{-1}(o)\cap P_1$. 
 Since $X$ is non-linearizable at $o$, in the coordinate $\sigma_p$, we may write
$$X_x=Ax+\langle w,x\rangle x,\ 0\neq w\notin \mathrm{Im}A^T.$$
Let $\nabla_p$ be the connection induced by the local section $\exp_p(\mathfrak{g}_{-1})$ at $p$. Then the type (2,1)-tensor $(\nabla_p-\nabla)$ vanishes at $o$. Thus we can assume that $\nabla$ is $\nabla_p$  in this proof. In the normal coordinates of $\nabla$ at $o$, denote $\overline{\Gamma^k_{i,j}}$ and $\Gamma^k_{i,j}$ the Christoffel symbols of $\nabla$ and $\nabla_{t_0}$, respectively. It follows that $\overline{\Gamma^k_{i,j}}(o)=\Gamma^k_{i,j}(o)=0$, because of $(\eta_{t_0})_o=0$. By calculations of the proof of Theorem $\ref{thm:kobayashi}$ of Nagano and Kobayashi in $\cite{Nagano}$, the exponential maps of $\nabla$ and $\nabla_{t_0}$ at $o$ have the same 2-jet. Denote $\exp^{\nabla}_o$ and $\exp^{\nabla_{t_0}}_o$ the exponential maps of $\nabla$ and $\nabla_{t_0}$ at $o$, respectively. Note $\sigma_p$ is a normal coordinate of $\nabla$ at $o$.
In the coordinate $\sigma_p$, choose $w_k\in \mathrm{Ker}A$ with $\langle w_k,w\rangle\neq 0$. Then in the coordinate $\sigma_p$, the curve $\gamma(s)=sw_k$ is a non-trivial parametrized geodesic of $\nabla$. Then There exists some $s_0>0$ such that $\gamma^{t_0}(s)=\phi^{-t_0}\circ\gamma(s)$ is well defined for $\vert s\vert<s_0$. Note that $w_k\in \mathrm{Ker}A$ implies the flow $\phi^t$ preserves the unparametrized geodesic $\gamma$. Because $\langle w, w_k\rangle \neq 0$, we have in $\sigma_p$:
$$\gamma^{t_0}(s)=\phi^{-t_0}\circ \gamma(s)=\dfrac{s}{1+as}w_k,\ a\neq 0.$$
Then near $s=0$, we can define the function $$f(s):=(\gamma^{-1}\circ\gamma^{t_0})(s)=\dfrac{s}{1+as}.$$
It is a local diffeomorphism fixing $0\in\mathbb{R}$.
The map $\phi^{-t_0}$ takes geodesics of $\nabla$ to geodesics of $\nabla_{t_0}$, so $\gamma^{t_0}(s)$ is a geodesic for $\nabla_{t_0}$ such that
 $$(\gamma^{t_0})'(0)=(\gamma)'(0)=w_k.$$
  Near $s=0$, we have $$\gamma(s)=\exp^{\nabla}_o(sw_k),\ \gamma^{t_0}(s)=\exp^{\nabla_{t_0}}_o(sw_k).$$ The exponential maps $\nabla$ and $\nabla_{t_0}$ have the same 2-jets at $o$, so $\gamma(s)$ and $\gamma^{t_0}(s)$ have the same 2-jets at $s=0$. This implies the function $f(s)$ has a trivial 2-jet at $s=0$. But we have 
 $$\dfrac{d^2}{ds^2}\bigg\lvert_{s=0}f(s)=-2a\neq0.$$ 
   Thus we have a contradiction.
\end{proof}
\subsection{BM-structures and Degree of mobility}
\label{sec:deg}
In general, there is an affine bijection between the elements in a given projective class $[\nabla]$ on the manifold $M^n$ and the 1-forms on $M^n$. The latter is an infinite dimension vector space, and is hard to analyse. So our focus is to study the metrizable elements of $[\nabla]$, where $\nabla$ is a Levi-Civita connection. From now on, let $g$ be a semi-Riemannian metric on $M^n$, and denote $\nabla$ its Levi-Civita connection.\\
\\
For any metric $\overline{g}$ on $M$, the $g$-strength of $\overline{g}$ is defined to be the (1,1)-tensor $K_{\overline{g}}$ such that
$$\overline{g}(u,v)=g\left(\dfrac{K_{\overline{g}}^{-1}}{\vert det(K_{\overline{g}})\vert}\cdot u,v\right).$$
We define a map $\rho(g)$ from the space of metrics on $M$ to the space of  non-degenerate $g$-adjoint (1,1)-tensors on $M$ as follows:
$$\rho(g)(\overline{g})=K_{\overline{g}}.$$
Clearly $\rho(g)$ is a bijection from the metrics on $M$ to the non-degenerate $g$-adjoint (1,1)-tensors on $M$.\\
\\
Let $f:M^n\rightarrow N^n$ be a smooth embedding. Fix metrics $g_1$ on $M$ and $g_2$ on $N$, respectively. We define the linear map $\rho^f(g_1,g_2):T^{1,1}N\rightarrow T^{1,1}M$ by 
$$\rho^f(g_1,g_2)(T)=f_*T\circ (\rho(g_1)(f^*g_2)).$$
Analogous to Fact 2.1 of $\cite{zig}$, if $g_2'$ is a metric on $N$, we have
$$\rho^f(g_1,g_2)(\rho(g_2)(g_2'))=\rho(g_1)(f^*g_2').$$
The map defined above is multiplicative in the following sense: Let $f_1:N_1\rightarrow N_2$, and $f_2:N_2\rightarrow N_3$ be smooth embeddings. Fix metrics $g_i$ on $N_i$, then we have
\begin{eqnarray}
\label{eqn3}
\rho^{f_1 \circ f_2}(g_1,g_3)=\rho^{f_1}(g_1,g_2)\circ \rho^{f_2}(g_2,g_3).
\end{eqnarray}
To proceed, we need the following definitions from Section 2 of $\cite{matrem}$:
\begin{definition}
\label{def:BM}
Suppose $g$ is a metric on $M^n$, the space of BM-structures on $M$ for $g$, denoted as $B(M,g)$, is the space of $g$-adjoint (1,1)-tensors on $M$ satisfying the following linear PDE, $\forall u,v,w\in T_xM,\ \forall x\in M$:
\begin{eqnarray}
\label{eqn4}
g((\nabla_w K)u,v)=\dfrac{1}{2}(d(trK)(u)g(v,w)+d(trK)(v)g(u,w)).
\end{eqnarray}
The degree of mobility of $g$ on $M^n$, denoted as $D(M^n,g)$, is the dimension of the vector space $B(M^n,g)$.
\end{definition}
\noindent
According to Equation (7)-(9) of $\cite{matpse}$, the non-degenerate elements of $B(M,g)$ are exactly the $g$-strength of the metrics projectively equivalent to $g$ on $M$. Equation \eqref{eqn4} is finite-type by Remark 5 of $\cite{matpse}$, so the solutions on each connected component are uniquely determined by the k-th jet at a single point for some $k\in\mathbb{N}$. Then we always have $D(M^n,g)<\infty$. In fact, according to Section 3 of $\cite{metric}$ , $[\nabla]$ defines a linear connection on some vector bundle $VM\simeq\bigodot^2 TM\oplus TM\oplus C^{\infty}(M)$. By Theorem 3.1 of $\cite{metric}$, solutions to Equation \eqref{eqn4} are 1-1 correspondence with parallel sections on $VM$. Hence, if $M^n$ is connected, we have $D(M^n,g)$ is at most the rank of $VM$: $$D(M^n,g)\leq \dfrac{(n+1)(n+2)}{2}.$$\\
\\We give a brief review of the Splitting Lemma used later in this paper. Given any $K\in B(M^n,g)$, denote $\chi(t)$ its characteristic polynomial in $t$. We say $\chi(t)$ admits an admissible factorization at $x\in M$, if $\chi(x)(t)=\chi_1(x)(t)\cdot\chi_2(x)(t)$, where $\chi_1(x)(t)$ and $\chi_2(x)(t)$ are non-constant polynomials in $t$ such that $\chi_1(x)(t)=0$ and $\chi_2(x)(t)=0$ have no common root. Since the eigenfunctions of $K$ can be chosen continuously, we have $\chi(t)$ admits such an admissible factorization on some neighbourhood $U_x$ of $x$. Then the following Splitting Lemma from $\cite{splitting}$ allows us to write the pair $(g,K)$ in block diagonal forms on $U_x$.
\begin{lemma}[Matveev, Bolsinov $\cite{splitting}$]
\label{lemma:splitting}
Suppose the characteristic polynomial $\chi(t)$ of $K\in B(M^n,g)$ admits an admissible factorization $\chi(t)=\chi_1(t)\chi_2(t)$ on some neighbourhood $U_x$ of $x$. Then there are local coordinates $(x_1,\cdots,x_r,y_1,\cdots,y_{n-r})$ at $x$ such that the pair $(g,K)$ can written in the following block diagonal form:
$$g=\begin{bmatrix}
h_1\chi_2(K_1) && 0\\
0 && h_2\chi_1(K_2)
\end{bmatrix},\quad K=\begin{bmatrix}
K_1 && 0\\
0 && K_2
\end{bmatrix},$$
where the pairs $(h_1,K_1)$ and $(h_2,K_2)$ depend only on the $x_i$ and $y_j$ coordinates, respectively. In addition, let $E_1$ and $E_2$ be distributions spanned by $\lbrace \partial x_i\rbrace_{i=1}^r$ and $\lbrace \partial y_j\rbrace_{j=1}^{n-r}$, respectively. Then $K_i$ is a BM-structure with the characteristic polynomial $\chi_i(t)$ for $h_i$ on each integral submanifold of $E_i$, respectively.
\end{lemma}
From now, assume $M$ is connected. Let $U$ be an open subset of $M$. Then $\forall K\in B(M,g)$, we have $K\vert_U\in B(U,g)$. The following restriction map is injective, since $M$ is connected.
$$R_U:B(M,g)\rightarrow B(U,g),\quad K\mapsto K\vert_U.$$
We can view $B(M,g)$ as a linear subspace of $B(U,g)$. Suppose $X$ is a projective vector field for $(M^n,g)$, and denote $\phi^t$ the flow generated by $X$. Further assume that $\exists a>0$ such that $\phi^t(x)$ is defined for $\forall x\in U$, and $\forall t\in I=[-a,a]$. Then the flow $\phi^t$ induces a well defined 1-parameter family of maps $L_t:B(M,g)\rightarrow B(U,g)$ for $t\in I$ as follows. Fix any $x\in U$ and $t\in I$. Suppose $\overline{g}$ is a metric defined on some neighbourhood $V_t$ of $\phi^t(x)$ such that $\overline{g}$ and $g$ are projectively equivalent on $V_t$. Near $x$, we have $(\phi^t)^*\overline{g}$ is a metric projectively equivalent to $g$. Denote $K_t$ the $g$-strength of $(\phi^t)^*g$, so it is well defined on $U$. Then near $x$, the tensor $\rho^{\phi^t}(g,g)(K_{\overline{g}})=\phi^t_*(K_{\overline{g}})\circ K_t$ is a solution to Equation \eqref{eqn4}. For any $y\in M$, we can always choose a neighbourhood $U_y$ of $y$ such that $B(U_y,g)$ has a basis consisting of non-degenerate elements. This implies for each $t\in I$, $\rho^{\phi^t}(g,g)$ defines a linear map $L_t:B(M,g)\rightarrow B(U,g)$ by $L_t(K')=\rho^{\phi^t}(g,g)(K')$.\\
\\
 If we further assume that $D(U,g)=D(M,g)$, every $K'\in B(U,g)$ can be uniquely extended to an element in $B(M,g)$. To simplify the notation, define $B=B(M,g)$. Then one can take $L_t$ as a map $L_t:B\rightarrow B$ for each $t\in I$. A natural question to ask is whether $L_t$ can be extended to a 1-parameter subgroup of $GL(B)$. This leads to the following lemma.
\\
\begin{lemma}
\label{lemma:group}
Let $(M^n,g)$ be connected with a projective vector field $X$. Suppose $X$ vanishes at $o\in M$. Assume that $U$ with $D(U,g)=D(M,g)$ is a connected open set containing $o$ such that $\phi^t$ is defined on $U$ for $t\in I=[-a,a]$ for some $a>0$. Then the map $L_t:B\rightarrow B$ defined in the previous paragraph satisfies the following:
\begin{itemize}
\item
 $L_{t+s}=L_t\circ L_s$ for $t,s,t+s\in I$. 
 \item
 The representation $L_t:I\rightarrow GL(B)$ is continuous in $t$.
\end{itemize}
In other words, $L_t$ can be extended to a 1-parameter subgroup of $GL(B)$.
\end{lemma}
\begin{proof}
Fix any $K'\in B=B(M,g)$. For any $t\in I$, $L_t(K')$ is the unique element in $B(M,g)$ such that:
$$L_t(K')\vert_U=\phi^t_*(K')\circ K_t\in B(U,g).$$
 Note that given the embedding $\phi^t:U\rightarrow M$, we have on $U$:
$$L_t(K')\vert_U=\rho^{\phi^t}(g,g)(K').$$
The embedding $\phi^s:U\rightarrow M$ gives $$L_s(L_t(K'))\vert_U=\rho^{\phi^s}(g,g)(L_t(K')).$$
Because $X$ vanishes at $o$, there is some neighbourhood $U_o$ of $o$ such that $\phi^s(U_o)\subset U$. Then we get the following sequence of embeddings:
$$U_o\xrightarrow{\phi^s} U\xrightarrow{\phi^t} M.$$
Because $t,s,t+s\in I$, by Equation \eqref{eqn3}, we have on $U_o$:
\begin{align}
\label{eqn5-7}
L_s(L_t(K'))\vert_{U_o}&=\left(\rho^{\phi^s}(g,g)\circ \rho^{\phi^t}(g,g)\right)(K'),\\
&=\rho^{\phi^{t+s}}(g,g)(K'),\\
&=L_{t+s}(K')\vert_{U_o}.
\end{align}
Since $U$ is connected, any BM-structure on $U$ is uniquely determined by its k-th jet at $o$ for some $k\geq0$. Then $L_{t+s}(K')=L_s\circ L_t(K')$ on $U_o$ implies   $L_{t+s}(K')=L_s\circ L_t(K')$ on $M^n$.\\
\\
Next we show the representation $L_t:I\rightarrow GL(B)$ is continuous in $t$. Because $L_t$ is linear for each $t$ and $B$ is a finite dimensional vector space, it is sufficient to show for any fixed $K'\in B$, $L_t(K')$ is continuous in $t$. Fix a compact neighbourhood $V_o\subset U$ of $o$ and a basis $\lbrace K^i\rbrace$ for $B$. Then we can write $L_t(K')=\sum c_i(t)K^i$, where $c_i:I\rightarrow\mathbb{R}$. Since \eqref{eqn4} is of finite type, $\lbrace K^i\rbrace$ are linearly independent over  $V_o$. We have on $U\supset V_o$, $L_t(K')=\phi^t_*(K')\circ K_t$. Then for any fixed $t_0\in I$, as $t\rightarrow t_0$, we have $L_t(K')\rightarrow L_{t_0}(K')$ uniformly on $V_0$. Then for each $i$, as $t\rightarrow t_0$, we have $c_i(t)\rightarrow c_i(t_0)$. This proves the continuity of $L_t:I\rightarrow GL(B)$. This implies $L_t$ can be extended to a continuous map defined on $\mathbb{R}$ with $L_t\circ L_s=L_{t+s}$. Hence the image of $\mathbb{R}^2$ under the map $(L_t,Id)$ is a closed subgroup in $GL(B)\times \mathbb{R}$. It follows that $L_t$ can be extended to a 1-parameter subgroup of $GL(B)$.
\end{proof}
\noindent
The following shows the neighbourhood $U$ in Lemma \ref{lemma:group} always exists.
\begin{lemma}
\label{lemma:exists}
Let $(M^n,g)$ be a connected manifold. Suppose $X$ is a projective vector field for $g$ vanishing at $o\in M$. Then there exists a connected open set $U$ containing $o$ such that $D(U,g)=D(M^n,g)$, and $\exists a>0$ such that $\phi^t$ is well defined on $U$ for $t\in I=[-a,a]$.
\end{lemma}
\begin{proof}
Define the following sets:
$$S_i=\lbrace x\in M: \phi^t(x)\ \textit{is well defined for}\ t\in[-\dfrac{1}{i},\dfrac{1}{i}] \rbrace.$$
Without loss of generality, we can assume $o\in Int(S_i)$ for all $i$. Let $U_i$ be the component of $Int(S_i)$ containing $o$. Since each $U_i$ is open and connected, it is also path connected. Given any $x\in U_i$, let $\gamma_x$ be a curve in $U_i$ joining $o$ and $x$. Then clearly $\gamma_x \subset Int(S_{i+1})$. It follows that $U_i\subset U_{i+1}$. Similarly, given any $x\in M$, we can choose a curve $\gamma_x'$ in $M$ joining $o$ and $x$. Then there exists $\epsilon>0$ and a neighbourhood $U_{\epsilon}$ of $\gamma_x'$ such that $\phi^t$ is well defined on $U_{\epsilon}$ for $t\in [-\epsilon,\epsilon]$. It follows that $x\in U_i$ for some $i$, hence $\bigcup_{i=1}^{\infty} U_{i}=M$. We have an increasing sequence of open sets containing $o$:
$$o\in U_1\subset U_2\subset \cdots,\quad \bigcup_{i=1}^{\infty}U_i=M.$$
Because each $U_i$ is connected, the restriction map gives a sequence of injective linear maps:
$$B(U_1,g)\xleftarrow{r_1} B(U_2,g)\xleftarrow{r_2} \cdots$$
We have $D(U_i,g)\geq D(M,g)$, and $D(U_1,g)< \infty$. It follows that there exists some $i_0$ such that $r_j:B(U_{j+1},g)\rightarrow B(U_j,g)$ are linear isomorphisms for all $j\geq i_0$. Then any $\tilde{K}\in B(U_{i_0},g)$ can be uniquely extended to an element in $B(U_j,g)$ for all $j\geq i_0$. Because a BM-structure on a connected manifold is uniquely determined by its finite jet at some point, we have $\tilde{K}$ can be extended to an element in $B(M,g)$. Then we get $D(U_{i_0},g)=D(M,g)$. This completes the proof.
\end{proof}
\noindent
Let $U$ be constructed by the lemma above. The map $L_t$ can be extended to a 1-parameter subgroup of $GL(B)$, also denoted as $L_t$. By the following, this construction is in fact coherent.
\begin{corollary}
\label{cor:2.1.2}
Let $X$ be a projective vector field for $(M,g)$ vanishing at $o$. Suppose $M$ is connected. Let $U$,$I$, and $L_t$ be constructed as above. Given any $t_0\in \mathbb{R}$, there exists some neighbourhood $V_{t_0}$ of $o$ such that $\phi^t$ is well defined for $\vert t\vert\leq \vert t_0\vert$, and $L_{t_0}(K')\vert_{V_{t_0}}=\phi^{t_0}_*(K')\circ K_{t_0}$ on $V_{t_0}$.
\end{corollary}
\begin{proof}
Without loss of generality, assume $t_0>0$. Let $U$, $I$ be the same as in Lemma \ref{lemma:group}, and $t_0=nt_1$ with $t_1\in I$. Given any $K'\in B\simeq B(U,g)$ and $t\in I$, there is some neighbourhood  $V_t$ of $o$ such that $\phi^t(V_t)\subset U$. In particular, we have $L_{t_1}(K')\vert_{V_{t_1}}=\phi^{t_1}_*(K')\circ K_{t_1}$. Assume there is some neighbourhood $V_{mt_1}\subset U$ of $o$ such that $\phi^s(V_{mt_1})$ is defined for $s\in [-mt_1,mt_1]$ such that
 $$L_{mt_1}(K')\vert_{V_{mt_1}}=\phi^{mt_1}_*(K')\circ K_{mt_1}.$$
We can choose some $V_{(m+1)t_1}$ such that $$o\in V_{(m+1)t_1}\subset V_{mt_1}\subset U,\quad \phi^{t'}(V_{(m+1)t_1})\subset V_{mt_1}\ for \ t'\in I.$$
Then $\phi^s$ is well defined on $V_{(m+1)t_1}$ for $s\in [-(m+1)t_1,(m+1)t_1]$. This implies on $V_{(m+1)t_1}$:
\begin{align}
\label{eqn8-11}
L_{(m+1)t_1}(K')\vert{V_{(m+1)t_1}}&=L_{t_1}(L_{mt_1}(K'))\vert{V_{(m+1)t_1}},\\
&=\phi^{t_1}_*(L_{mt_1}(K'))\circ K_{t_1},\\
&=\phi^{t_1}_*(\phi^{mt_1}_*(K')\circ K_{mt_1})\circ K_{t_1},\\
&=\phi^{(m+1)t_1}_*(K')\circ K_{(m+1)t_1}.
\end{align}
By induction, we have on $V_{t_0}=V_{nt_1}$, $L_{t_0}(K')\vert_{V_{t_0}}=\phi^{t_0}_*(K')\circ K_{t_0}$.
\end{proof}
\section{Local results and general theory when $D(M,g)$ is $2$}
\label{sec:localresult}
Let $(M^n,g)$ be a connected manifold with $D(M,g)=2$. Let $X$ be a projective vector field for $g$ with a singularity $o$. Denote $\phi^t$ the flow generated by $X$. Suppose $X$ is not linearizable at $o$. Then $L_t$ is a 1-parameter subgroup of $GL(B)\simeq GL_2(\mathbb{R})$. By Corollary \ref{cor:2.1.2}, for any fixed $t\in \mathbb{R}$, on some neighbourhood $V_t$ of $o$, we have
 $$L_t(K')= \phi^t_*(K')\circ K_t.$$ 
In particular on $V_t$, we have $L_t(Id)=K_t$. By Corollary \ref{cor:2.1.1}, for any $t\neq 0$, the metrics $g_t$ and $g$ are not affine equivalent on any neighbourhood of $o$. This implies the eigenfunctions of $K_t$ are not all constant on any neighbourhood of $o$. Otherwise by Equation \eqref{eqn4}, we get $\nabla K_t=0$ near $o$, then $g_t$ and $g$ are affine equivalent near $o$. The group $L_t$ is elliptic if and only if its action on $\mathbb{P}(B)$ is periodic. Suppose $L_t$ is elliptic, then $\exists t_0\neq 0$ such that $K_{t_0}=L_{t_0}(Id)=rId$ with $r\neq 0$. Thus $L_t$ cannot be an elliptic 1-parameter subgroup of $GL(B)$. We can prove $L_t$ is in fact parabolic:
\begin{theorem}
\label{thm:parabolic}
Let $(M^n,g)$ be a connected semi-Riemannian manifold with $D(M,g)=2$. Let $X$ be a projective vector field for $g$ vanishing at $o$. Suppose $X$ is not linearizable at $o\in M$, then $L_t$ is a 1-parameter parabolic subgroup of $GL(B)$.
\end{theorem}
\noindent
The idea of the proof of Theorem $\ref{thm:parabolic}$ follows from $\cite{zig}$ by Zeghib. Before proving the theorem, we make the following observation. Let $U,I,L_t$ be as before. Fix any $t_0\neq 0$, we have $\lbrace L_{t_0}(Id), Id\rbrace$ is a basis for $B$. Write $\overline{K}$ for $L_{t_0}(Id)$ for simplicity. Analogous to Section 4.2.1 of $\cite{zig}$, we can write
\begin{align}
\label{eqn12}
L_{t_0}(\overline{K})=\alpha\overline{K}+\beta Id,\quad L_{t_0}(Id)=\overline{K}.
\end{align}
As in Section 4.2 of $\cite{zig}$, define the associated Mobius map
 $$T:\widehat{\mathbb{C}}\rightarrow\widehat{\mathbb{C}},\quad T(z)=\dfrac{\alpha z+\beta}{z}.$$
Now further assume $t_0\in I$, then we have $\overline{K}\vert_U=K_{t_0}$. Thus for $x\in U$, we have
\begin{align}
\label{eqn13}
(\alpha\overline{K}+\beta Id)_x=\left(L_{t_0}(\overline{K})\right)_x=(\phi^{t_0}_*(\overline{K})\circ K_{t_0})_x=(\phi^{t_0}_*(\overline{K}))_x\circ \overline{K}_x.
\end{align}
For $x\in U$, we have $\det(\overline{K}_x)=\det((K_{t_0})_x)\neq 0$. This give the following: $$(\phi^{t_0}_*(\overline{K}))_x=(D\phi^{t_0}_x)^{-1} \overline{K}_{\phi^{t_0}(x)} D\phi^{t_0}_x=(\alpha Id+\beta \overline{K}^{-1})_x.$$ 
Note the right hand side is $(T(\overline{K}))_x$. It follows that $\overline{K}_{\phi^{t_0}(x)}$ and $(T(\overline{K}))_x$ have the same Jordan form. For $x\in U$, we get
\begin{align}
\label{eqn14}
T(\mathrm{Spec}(\overline{K}_x))=\mathrm{Spec}\left(\overline{K}_{\phi^{t_0}(x)}\right).
\end{align}
To prove Theorem $\ref{thm:parabolic}$, we also need the following lemma.
\begin{lemma}
\label{lemma:eigenvalue}
Suppose $L_t$ is induced by a projective vector field admitting a non-linearizable vanishing point $o\in M$. Fix any $t_0\neq 0$, and define $\overline{K}$ and $T$ as before. Then $L_t$ defines a non-trivial 1-parameter parabolic or hyperbolic subgroup of $PGL(B)$ acting on $\mathbb{P}(B)$. Its fixed set on $\mathbb{P}(B)$ is exactly the following:
$$D_o=\lbrace [\overline{K}-rId]:r\in \mathrm{Spec}((\overline{K})_o)\cap \mathbb{R}\rbrace$$
Moreover, the fixed set of the Mobius map $T$ on $\widehat{\mathbb{C}}$ is exactly $\mathrm{Spec}(\overline{K}_o)$.
\end{lemma}
\begin{proof}
As we have already noted in the first paragraph of this section $L_t$ is either hyperbolic or parabolic. Then for any $t_0\neq 0$, the fixed set of $L_{t_0}$ on $\mathbb{P}(B)$ is the fixed set of $L_t$ on $\mathbb{P}(B)$. It is clearly non-empty. For any fixed $t_0\neq 0$, by Corollary \ref{cor:2.1.2}, there is a neighbourhood $V$ of $o$ such that $$L_{t_0}(K')\vert_V=\phi^{t_0}_*(K')\circ K_{t_0},\ \forall K'\in B.$$ Then $(L_{t_0}(K'))_o$ is degenerate if and only if $(K')_o$ is degenerate. For $K\in B$, we have $K\in D_o$ if and only if $K$ is degenerate at $o$. This implies $L_{t_0}$ takes $D_o\subset \mathbb{P}(B)$ to itself. Because $D_o$ is a finite discrete subset of $\mathbb{P}(B)$, we have $L_t$ fixes all elements in $D_o$.\\
\\
Suppose that there is some $[\overline{K}-r_0Id]\notin D_o$ fixed by $L_t$, and we seek a contradiction. Let $K^1=\overline{K}-r_0Id$, then $L_t(K^1)=e^{ct}K^1$, for some $c\in \mathbb{R}$. Note $K^1$ is non-degenerate near $o$. Then $K^1$ defines a metric $g_{K^1}$ projectively equivalent to $g$ on some neighbourhood $V_o\subset U$ of $o$. Because $L_t(K^1)\vert_U=\phi^t_*(K^1)\circ K_t$ for $t\in I$, we have $X$ is a homothetic vector field for $g_{K^1}$. This is impossible. Also note that $L_t$ does not fix the line $[Id]$, otherwise $X$ is a homothetic vector field for $g$. This proves the fixed set of $L_t$ on $\mathbb{P}(B)$ is exactly $D_o$\\
\\
For any fixed $t_0\neq 0$, the associated Mobius map is of the form $T(z)=\dfrac{\alpha z+\beta}{z}$. Under the basis $\lbrace \overline{K},Id\rbrace$, $L_{t_0}$ has the following matrix representation:
$$\begin{bmatrix}
\alpha &&1\\
\beta && 0
\end{bmatrix}.$$
Denote $F(T)$ the fixed set of $T$ on $\widehat{\mathbb{C}}$. Then $L_{t_0}$ fixes exactly $D_o$ implies $F(T)\cap\mathbb{R}$ is exactly $\mathrm{Spec}(\overline{K}_o)\cap \mathbb{R}$. Because $L_{t_0}$ is non-elliptic, it fixes some line $[\overline{K}-r_0Id]\in D_o$. It follows that $\beta=-r_0(\alpha-r_0)$ with $r_0\in\mathbb{R}$. Then the equation $z^2=\alpha z+\beta$ has 1 or 2 distinct real root. In either case, we have $F(T)$ is a subset of $\mathbb{R}$, so $F(T)=\mathrm{Spec}(\overline{K}_o)\cap \mathbb{R}$. In addition, the finite subsets of $\widehat{\mathbb{C}}$ preserved by $T$ are subsets of $F(T)$. According to Equation \eqref{eqn14}, we have $\mathrm{Spec}((\overline{K})_o)$ is a finite set preserved by $T$. It follows that $F(T)=\mathrm{Spec}((\overline{K})_o)$. This completes the proof.
\end{proof}
\noindent
Now we can prove Theorem $\ref{thm:parabolic}$.
\begin{proof}[Proof of Theorem $\ref{thm:parabolic}$]
The general scheme of this proof is as follows. First, we use the normal form of $X$ given in Lemma $\ref{lemma:linear}$ to obtain the dynamics of $\phi^t$ on some special geodesic curve $\gamma$. Assume $L_t$ is hyperbolic and fix some $t_0\neq 0$. The Splitting Lemma allows us to write $(g,K_{t_0})$ in block diagonal forms. Using this and the hyperbolicity of $T$, we study the behaviour of the eigenfunctions of $K_{t_0}$ along $\gamma$. The dynamics of $\phi^{t_0}$ along $\gamma$ and the dynamics of the associated Mobius map $T$ are related by eigenfunctions of $K_{t_0}$ as in Equation $(\ref{eqn14})$. We use this to derive a contradiction.\\
\\
By Lemma \ref{lemma:eigenvalue}, $L_t$ is either hyperbolic or parabolic. Suppose $L_t$ is hyperbolic. Choose $0\neq t_0\in I$, then $K_{t_0}$ is the $g$-strength of $g_{t_0}$ on $U$. Denote $\nabla$ the Levi-Civita connection for $g$. Let $P=P(\nabla)$ be the projective Cartan bundle for $\nabla$.  Then $\nabla$ induces a $GL_n$ sub-bundle $\Gamma$ of $P$. Choose $p\in \Gamma \cap \pi^{-1}(o)$. The section given by $\exp_p(\mathfrak{g_{-1}})$ locally defines a torsion-free affine connection $\overline{\nabla}\in[\nabla\vert_V]$ on some neighbourhood $V$ of $o$. Let $\sigma_p$ be a projective normal coordinate of $P$ with respect to $p$. Clearly by Theorem $\ref{thm:kobayashi}$, $\sigma_p$ is a normal coordinate of $\overline{\nabla}$ at $o$. Because $X$ is not linearizable at $o$, by Lemma \ref{lemma:linear}, $(\sigma_p)^{-1}_*X$ has the following form:
$$X_x=Ax+\langle w,x\rangle x,\quad w\notin Im(A^T).$$
Choose $v\in \mathrm{Ker}A$ such that $\langle w,v\rangle\neq 0$. In the local coordinate $\sigma_p$, there exists $a\neq0$ and $\epsilon>0$ such that
\begin{align}
\label{eqn15}
\phi^{t}(yv)=\left(\dfrac{y}{1+tay}\right)v,\ y\in(-\epsilon,\epsilon),\ t\in I.
\end{align}
Let $\gamma(s)$ and $\gamma(s(y))$ be geodesics with initial vector $(\sigma_p)_*v$ for $\nabla$ and $\overline{\nabla}$, respectively. Denote $E:T_oM\rightarrow M$ and $\overline{E}:T_oM\rightarrow M$ the exponential maps for $\nabla$ and $\overline{\nabla}$ at $o$, respectively. From Theorem $\ref{thm:kobayashi}$ by Nagano and Kobayashi, we have $J^2(E)(0)=J^2(\overline{E})(0)$, because $p\in\Gamma\cap \pi^{-1}(o)$. Then we obtain
\begin{align}
\label{eqn16}
\dfrac{ds}{dy}(0)=1,\quad \dfrac{d^2s}{dy^2}(0)=0.
\end{align}
Note that $\phi^t$ preserves the unparametrized geodesic given by $\gamma$. Then for small $s$, define a parametrized family of functions $\tau_t$ with $\tau_t(0)=0$ for $t\in I$ by the following:
$$\phi^{t}\circ \gamma(s)=\gamma(\tau_t(s)).$$
Let $\tau=\tau_{t_0}$ for simplicity. From Equation \eqref{eqn15}, we also have $\dfrac{d\tau}{ds}(0)=1$. As in Equation (5) of $\cite{matpse}$, define the function:
$$\psi(s)=-\dfrac{1}{2}\log(\det(K_{t_0}))(\gamma(s)).$$
Then for small $s$, we have by Equation (2) and (3) of $\cite{matpse}$:
$$\dfrac{d\psi}{ds}=\dfrac{1}{2} \dfrac{d}{ds}(\log(\dfrac{d\tau}{ds})).$$
It follows that $\dfrac{d\psi}{ds}(0)=\dfrac{1}{2}\dfrac{d^2\tau}{ds^2}(0)$. According to Lemma \ref{lemma:eigenvalue}, $\mathrm{Spec}((K_{t_0})_o)=\lbrace \lambda_u,\lambda_b\rbrace \subset \mathbb{R}$. Here $\lambda_u,\lambda_b$ are the unstable and stable fixed point of the associated Mobius map $T(z)=\dfrac{\alpha z+\beta}{z}$, respectively.
We can apply the Splitting Lemma (Lemma $\ref{lemma:splitting}$). On some neighbourhood $V'\subset V$ of $o$, there is a smooth local coordinate in which $K_{t_0}$ can be written in the following block-diagonal form:
$$K_{t_0}=\begin{bmatrix}
K_u && 0\\
0 && K_b
\end{bmatrix},\quad \mathrm{Spec}((K_u)_o)=\lbrace \lambda_u \rbrace,\ \mathrm{Spec}((K_b)_o)=\lbrace \lambda_b \rbrace.$$
 We may choose $V'$ small enough so that $\mathrm{Spec}(K_u)\vert_{V'}\subset D_u$, and $\mathrm{Spec}(K_b)\vert_{V'}\subset D_b$. Here $D_u, D_b$ are 2 disjoint disks in $\mathbb{C}$ centered at $\lambda_u,\lambda_b$, respectively. It follows that
\begin{align}
\label{eqn17}
\psi(s)=-\dfrac{1}{2}\left( \log(det(K_u))(\gamma(s))+\log(det(K_b))(\gamma(s))\right).
\end{align}
Define $f_u(s)=\det(K_u)(\gamma(s))$, and $f_b(s)=\det(K_b)(\gamma(s))$. Without loss of generality, let us assume $t_0a>0$. From Equation \eqref{eqn15}, for small $s>0$, we have $\tau(s)<s$, and $\phi^{mt_0}(\gamma(s))\rightarrow o $ as $m\rightarrow +\infty$. If we choose the eigenfunctions of $K_u$ and $K_b$ to be continuous on $V'$, then it can be shown that the eigenfunctions of $K_u$ have to be constant on $\gamma(s)$ for small $s>0$. Suppose this is not the case. Let $\tilde{k_u}$ be an eigenfunction of $K_u$, and write $k_u(s)=\tilde{k_u}(\gamma(s))$. Then there is some $s_0>0$ such that $\gamma([0,s_0])\subset V',\ k_u(s_0)\neq \lambda_u$. 
$$\gamma([0,s_0])\subset V'\subset V\Longrightarrow \phi^{t_0}\circ \gamma([0,s_0])\subset \gamma([0,s_0]).$$
 The map $T$ is continuous on $\widehat{\mathbb{C}}$. Therefore, $T^m\circ k_u:[0,s_0]\rightarrow \widehat{\mathbb{C}}$ is a continuous map for each $m$. For large $m$, we have $T^m(k_u(s_0))\in D_b$. On the other hand, for any $s'\in [0,s_0]$ we have $$T^m(k_u(s'))\in Spec\left((K_{t_0})(\phi^{mt_0}\circ \gamma(s'))\right)\subset D_u \cup D_b.$$
 Because $T^m(k_u(0))=\lambda_u$ for all $m$, we have $T^m\circ k_u([0,s_0])$ is not connected for large $m$. This contradicts the continuity.\\
\\
The above implies $f_u(s)$ is constant for small $s\geq 0$. Similarly, we can prove $f_b(s)$ is constant for small $s\leq 0$. From Equation \eqref{eqn17}, we have $\dfrac{d\psi}{ds}(0)=0$. It follows that
$$\dfrac{d^2\tau}{ds^2}(0)=0.$$
Define the Mobius map $\widehat{T}(y)=\dfrac{y}{1+t_0ay}$. From Equation \eqref{eqn15}, we have near $0$:
$$\tau\circ s(y)=s\circ \widehat{T}(y).$$
By Equation \eqref{eqn16}, we have $J^2(\tau)(0)=J^2(\widehat{T})(0)$. This gives $\dfrac{d^2}{dy^2}(\widehat{T})(0)=0$, which is clearly impossible because $t_0a\neq 0$. We obtain a contradiction. Hence $L_t$ can only be a 1-parameter parabolic subgroup of $GL(B)$.
\end{proof}
\section{Global results when $(M^n,g)$ is closed or Riemannian}
\label{sec:globalresult}
\subsection{Result for the case $g$ is Riemannian, proof of Theorem \ref{thm:rem}}
\label{sec:rem}
In this section, we give the proof of Theorem $\ref{thm:rem}$ stated in the introduction.\\
\\
Before we prove the theorem, we make the following observations. Let $(M^n,g)$ with $n\geq 3$ be a connected Riemannian manifold with $D(M^n,g)=2$. Then $\forall K'\in B(M,g)$, $K'$ is real diagonalizable, because it is a self-adjoint operator for the Riemannian metric $g$. Let $U,I,L_t$ be as before. We know from Theorem $\ref{thm:parabolic}$ that $L_t$ is a 1-parameter parabolic subgroup. Fix any $0\neq t_0\in I$, by Lemma \ref{lemma:eigenvalue}, $(K_{t_0})_o$ has only 1 real eigenvalue $\lambda>0$. We have $(K_{t_0})_o=\lambda Id$. Because $X$ is not linearizable at $o$, by Lemma \ref{lemma:linear}, $(D\phi^t)_o$ fixes some non-zero $v\in T_oM$. It follows that $$g(v,v)=g_{t_0}(v,v)=\dfrac{1}{\det((K_{t_0})_o)}g((K_{t_0})^{-1}_ov,v).$$
 Then we have $\lambda=1$, and $(K_{t_0})_o=Id$. By Lemma \ref{lemma:eigenvalue}, the associated Mobius map for $L_{t_0}$ is $T(z)=\dfrac{2z-1}{z}$.\\
 \\
 Now we are ready to prove Theorem $\ref{thm:rem}$.

\begin{proof}[Proof of Theorem $\ref{thm:rem}$]
First we prove $D(M^n,g)\geq 3$. Suppose $D(M,g)=2$, and we try to obtain a contradiction.\\
\\
Let $U,I,L_t$ be constructed as before. Fix some $0<t_0\in I$. We have $(\phi^t)^*g(o)=g(o)$ for all $t\in I$. This implies $(D\phi^t)_o$ is a 1-parameter subgroup of $SO(g)$ at $o$. By Remark \ref{rmk1}, we can choose $p\in \pi^{-1}(o)$ such that in the projective normal coordinate $\sigma_p$ of $P=P(\nabla)$ with respect to $p$, $X$ has the following form:
$$X_x=Ax+\langle w,x\rangle x,\quad A\in\mathfrak{so}(n),\ w=-e_1\in \mathrm{Ker}A.$$
Then in this local coordinate $\sigma_p$, the flow $\phi^t$ of $X$ has the following form:
\begin{align}
\label{eqn18}
\phi^t(x)=\dfrac{1}{1+tx_1}\left(e^{tA}x\right),\quad x=(x_1,\cdots,x_n).
\end{align}
Choose a convex neighbourhood $C$ of $o$ which lies in the image of the local coordinate $\sigma_p$. By Corollary 3 of $\cite{matrem}$,  for all $i\in \lbrace 1,\cdots,n-1\rbrace$, the eigenfunctions $\lambda_i$ of $K_{t_0}$ are globally ordered on $C$ in the following sense:
\begin{itemize}
\item
$\lambda_i(x)\leq \lambda_{i+1}(y)$ for all $x,y\in C$.
\item
If $\exists x\in C$ such that $\lambda_i(x)<\lambda_{i+1}(x)$, then $\lambda_i(y)<\lambda_{i+1}(y)$ for almost all $y\in C$.
\end{itemize}
At $o$, we have $\lambda_i(o)=1$ for all $i$. Note that $n\geq3$ implies $\lambda_2= \cdots= \lambda_{n-1}\equiv 1$ on $C$. Indeed it follows that for $n\geq3$, $\lambda_1(x)\leq \lambda_2(x)= 1$, and $\lambda_n(x)\geq \lambda_{n-1}(x)= 1$ for all $x \in C$.  We can show all eigenfunctions $\lambda_i$ have to be constant on $C$. In the coordinate $\sigma_p$, define the following subsets of $C$:
$$C^+=\lbrace x\in C:x_1>0\rbrace,\ C^-=\lbrace x\in C: x_1<0\rbrace.$$
If $\exists x_1\in C$ such that $\lambda_1(x_1)<1$, we can find $x_0\in C^+$ such that $\lambda_1(x_0)<1$, and $\phi^t(x_0)\in C^+$ for all $t\geq0$. Denote $\mathcal{D}$ the closure of the integral curve of $\phi^t(x_0)$ for $t\geq0$, then clearly $\mathcal{D}\subset C$. From Equation \eqref{eqn18}, we can see $\mathcal{D}$ is compact and connected. Hence $\lambda_1(\mathcal{D})$ is an interval $I_1=[d,1]$ with $d<1$. The eigenfunctions of $K_{t_0}$ are all positive on $U$, so we have $0<d<1$ and $0<\lambda_1(x)\leq 1$ $\forall x\in \mathcal{D}$. Because $T(z)=\dfrac{2z-1}{z}$ is monotonically increasing on $\mathbb{R}^+$, we have $T(\lambda_1(x))=\lambda_1(\phi^{t_0}(x))$ for all $x\in\mathcal{D}$. It follows that $$T([d,1])=T(\lambda_1(\mathcal{D}))=\lambda_1(\phi^{t_0}(\mathcal{D}))\subset \lambda_1(\mathcal{D})=[d,1],\ 0<d<1.$$
This is clearly impossible for the Mobius map $T(z)=\dfrac{2z-1}{z}$ as $T(d)<d$ for $0<d<1$. Hence $\lambda_1\equiv 1$ on $C$. Replacing $C^+$ with $C^-$, and $T$ with $T^{-1}$, respectively, we can show $\lambda_n\equiv 1$ on $C$. It follows that all eigenfunctions of $K_{t_0}$ are constant on $C$.\\
\\
If all eigenfunctions of $K_{t_0}$ are constant on $C$, the metrics $g_{t_0}$ and $g$ are affine equivalent on $C$. This is clearly impossible by Corollary \ref{cor:2.1.1}. It follows that $D(M,g)\neq 2$.\\
\\
Since $X$ is a projective vector field for $(M^n,g)$, according to Section 2.1 of $\cite{matrem}$, we have
$$K'=g^{-1}\mathcal{L}_Xg-\dfrac{1}{n+1}Tr(g^{-1}\mathcal{L}_Xg)\cdot Id \in B(M,g).$$
Then $D(M,g)=1$ implies that $X$ is a homothetic vector field for $g$, which is impossible. Hence we have $D(M,g)\geq 3$.\\
\\
When $n=3$, by Section 1.2 of $\cite{mike}$, the maximum degree of mobility of a 3-dimensional connected Riemannian manifold with non-constant curvature is 2. This completes the proof.
\end{proof}
\begin{remark}
\label{rmk2}
The conditions $n\geq 3 $, and $g$ is Riemannian are necessary in the proof. If $n=2$, one may end up with $\lambda_1=1$ on $C^+,\lambda_1<1$ on $C^-$, together with $\lambda_2=1$ on $C^-,\lambda_2>1$ on $C^+$. If $g$ is not Riemannian, $(K_{t_0})_o$ may not be the identity matrix. Besides, the global ordering of eigenfunctions of BM-structures can only be applied for Riemannian metrics
\end{remark}
\subsection{Global results when $(M^n,g)$ is closed, proof of Theorem $\ref{thm:closed}$}
\label{sec:closed}
In this section, we give the proof of Theorem $\ref{thm:closed}$ stated at the end of the introduction.
\begin{proof}[Proof of Theorem \ref{thm:closed}]
Since $X$ is not linearizable at $o$, we have $D(M,g)\geq 2$. First suppose $D(M,g)=2$, then $L_t$ is a 1-parameter parabolic subgroup by Theorem $\ref{thm:parabolic}$. This is in fact impossible by the following (We discovered that the argument below is analogous to part of Section 9.2 of $\cite{zig}$).
\\
\\
Because $L_t$ is parabolic, there exists $K\in B=B(M,g)$ such that 
$$L_t(Id)=e^{tb}(tK+Id),\ b\in \mathbb{R}.$$
$X$ is complete because $M$ is compact. Just fix $t=1$, then $L_1(Id)=e^b(K+Id)$ is the $g$-strength of $(\phi^1)^*g$ on $M$. Because $M$ is closed and connected, according to Theorem 6 of $\cite{local}$, all non-real eigenfunctions of $L_1(Id)$ are constant. It follows that all non-real eigenfunctions of $K$ are constant on $M$. On the other hand,  all real eigenfunctions of $K$ are identically zero. Otherwise, $\exists t_0\in \mathbb{R}$ such that $L_{t_0}(Id)=K_{t_0}$ is degenerate. Then all eigenfunctions of $K$ are constant. This implies $g_t$ and $g$ are affine equivalent for all $t\in \mathbb{R}$, which is impossible.
\\
\\
From above we have $D(M,g)\geq 3$. According to Corollary 5.2 of $\cite{closed}$, we have $g$ is Riemannian with positive constant sectional curvature.
\end{proof}



\end{document}